\documentclass[11pt]{article}
\usepackage[T1]{fontenc}
\usepackage{lmodern}
\usepackage{amsmath,amsthm,amssymb,amsfonts}
\usepackage{enumerate}
\usepackage{epsfig}
\usepackage[dvipsnames,svgnames,table]{xcolor}
\usepackage{pgf,tikz}
\usetikzlibrary{calc}
\usepackage{fullpage}
\usepackage[colorlinks=true,linkcolor=blue,urlcolor=blue,citecolor=red]{hyperref}
\usepackage{booktabs}
\usepackage[nameinlink]{cleveref}
\usepackage{enumitem}

\newtheorem{theorem}{Theorem}[section]
\newtheorem{proposition}[theorem]{Proposition}
\newtheorem{lemma}[theorem]{Lemma}
\newtheorem{corollary}[theorem]{Corollary}
\newtheorem{conjecture}[theorem]{Conjecture}

\theoremstyle{definition}

\theoremstyle{remark}

\let\emph\relax 
\DeclareTextFontCommand{\emph}{\bfseries\em}

\DeclareMathOperator{\radii}{radii}
\DeclareMathOperator{\trdeg}{tr~deg}

\tikzstyle{edge}=[line width=1.5pt,black]
\tikzstyle{vertex}=[fill=black,circle,inner sep=0pt, minimum size=5.5pt]

\title{Identifying contact graphs of sphere packings with generic radii}
\author{Sean Dewar\thanks{School of Mathematics, University of Bristol. E-mail: \texttt{sean.dewar@bristol.ac.uk}}
}

\begin{document}
\date{}
\maketitle

\begin{abstract}
	Ozkan et al.~\cite{meeraetal} conjectured that any packing of $n$ spheres with generic radii will be stress-free, and hence will have at most $3n-6$ contacts.
	In this paper we prove that this conjecture is true for any sphere packing with contact graph of the form $G \oplus K_2$,
	i.e., the graph formed by connecting every vertex in a graph $G$ to every vertex in the complete graph with two vertices.
	We also prove the converse of the conjecture holds in this special case:
	specifically, a graph $G \oplus K_2$ is the contact graph of a generic radii sphere packing if and only if $G$ is a penny graph with no cycles.
\end{abstract}



\section{Introduction}\label{sec:intro}

In our context a sphere packing will always refer to a finite set spheres in $\mathbb{R}^3$ with disjoint interiors,
and a circle packing will always refer to a finite set of circles in $\mathbb{R}^2$ with disjoint interiors.
A vital structure associated with any sphere packing (respectively, circle packing) $P=\{S_v :v \in V\}$ is its \emph{contact graph},
the (finite simple) graph $G=(V,E)$ where a distinct vertices $v,w$ are adjacent if and only if the corresponding spheres $S_v$ and $S_w$ are in contact (or equivalently, are tangent).
The circle packing theorem (see for example \cite{stephenson}) states that a graph is the contact graph of a circle packing if and only if the graph is planar,
the latter property being determinable in linear time \cite{HopcroftTarjan}.
In contrast,
determining whether a graph is the contact graph of a sphere packing is NP-Hard (see \Cref{sec:prelim-sphere} for more detail).
Even determining the maximum possible number of contacts for a packing of $n$ spheres (denoted here by $\alpha_n$) remains open.
The current best bounds for $\alpha_n$ when $n$ is sufficiently large are $6.30 n < \alpha_n < 6.96 n$;
see \cite{EKZ03} and \cite{Glas20} for the lower and upper bounds respectively.

In this paper we instead ask a slightly different question:
how many contacts can a sphere packing with randomly chosen radii have?
For example, while 5 spheres can be in mutual contact (and so $\alpha_5 = 10$),
the radius of the fifth sphere is uniquely determined by the four previous spheres.
The following simple heuristic argues that the number of contacts should in fact be less than half the value of $\alpha_n$:
Suppose that $P$ is a packing of $n-1$ spheres with $k$ contacts.
If we add a new sphere $S$ to $P$ then we would expect the new packing to have at most $k+3$ contacts,
since $S$ can only be in contact with more than 3 other spheres if its radius is carefully chosen to do so.
Hence we would expect that $n$ spheres with randomly chosen radii would have roughly $3n$ contacts between them.

This idea of ``random radii implies low contact number'' was first codified by Ozkan et al.~\cite{meeraetal} in what we refer to as the \emph{genericity conjecture for sphere packings}.
The conjecture states a somewhat stronger property than a bound on the number of contacts and uses the language of \emph{rigidity theory}, the study of the kinematic properties of discrete structures.
It claims that when the radii of a sphere packing are \emph{generic} (i.e., form an algebraically independent set\footnote{There is a technical caveat here: whenever we refer to a finite indexed set $\{a_i: i \in I\} \subset \mathbb{R}$ being algebraically independent, we also assume that the set has $|I|$ elements (i.e., each element $a_i$ is distinct). This is to avoid cases such as $I= \{1,2\}$ and $a_1=a_2 = \pi$, as while the set $S=\{a_i : i \in I\} = \{\pi\}$ is technically algebraically independent, its indexed elements satisfy the polynomial equation $a_1-a_2=0$.}),
the packing will be \emph{stress-free} when considered as a \emph{bar-and-joint framework} (see \Cref{sec:prelim-sphere} for rigorous definitions of these concepts).
Using a famous result of Maxwell -- namely, that any stress-free bar-and-joint framework in $\mathbb{R}^3$ with $n$ joints has at most $3n-6$ bars \cite{maxwell} -- we obtain the following statement for the conjecture.

\begin{conjecture}[Genericity Conjecture for Sphere Packings]\label{conj:main}
	Let $G=(V,E)$ be the contact graph of a sphere packing $P$ with generic radii.
	Then $P$ is stress-free,
	and hence $|E| \leq 3|V|-6$ if $|V| \geq 3$.
\end{conjecture}

The 2-dimensional analogue of \Cref{conj:main} (\Cref{t:genconj2}) was proven by Connelly, Gortler and Theran \cite{congortheran2019} using techniques linked to the Cauchy–Alexandrov stress lemma \cite[Lemma 5.3]{gluck}.
Their proof requires that the set of all circle packings with contact graph $G=(V,E)$ form a smooth $(3|V|-|E|)$-dimensional semi-algebraic set.
Unfortunately the 3-dimensional analogue of this statement -- in that the set of all sphere packings with contact graph $G=(V,E)$ forms a smooth $(4|V|-|E|)$-dimensional semi-algebraic set -- is not true in general since there exist contact graphs of sphere packings with more than $4|V|$ edges; see \Cref{c:pennysphere} for a method for constructing an infinite (and also very relevant) family of such sphere packing contact graphs.
It is open whether this method can somehow be adapted for spherical packings,
although it is likely any such adaptation would require a much better understanding of the possible types of contact graphs for general sphere packings.
Similar 2-dimensional analogues of \Cref{conj:main} were proven by the author for homothetic packings of certain types of convex bodies \cite{dew21} and for homothetic packings of squares \cite{dew22}.

In this paper we shall prove the following stronger variation of \Cref{conj:main} for a specific family of graphs.
Here we recall that a graph is a \emph{penny graph} if it is the contact graph of a packing of unit radius circles,
and the join $G_1 \oplus G_2$ of a pair of graphs $G_1,G_2$ is the graph formed by connecting every vertex in $G_1$ to every vertex in $G_2$.

\begin{theorem}\label{t:main}
	A graph $G \oplus K_2$ is the contact graph of a sphere packing with generic radii if and only if $G$ is both a forest and a penny graph.
	Furthermore,
	any sphere packing with generic radii and contact graph $G \oplus K_2$ is stress-free.
\end{theorem}

%

It was asked in \cite{congortheran2019} whether every $(2,3)$-sparse graph -- a graph where every induced subgraph with $k \geq 2$ vertices has at most $2k-3$ edges -- that is planar is the contact graph of a generic radii circle packing.
If true, this would provide a polynomial-time algorithm for determining whether a graph is the contact graph of a generic radii circle packing \cite{JH97},
yet again contrastingly sharply with the spherical case.
Interestingly, it is known that the analogous statement of this conjecture is true for ``most'' choices of centrally symmetric convex body in the plane, even when applied to a larger class of graphs; see \cite{dew21} for more details.

Our approach to proving \Cref{t:main} is essentially a simple argument with multiple technical caveats.
We break the proof into three separate parts.

First, we prove that if a generic radii sphere packing has contact graph $G \oplus K_2$ then $G$ is a forest (\Cref{l:main}).
Our technique for doing so is to reduce the sphere packing to an embedding of the graph $G$ in $\mathbb{R}^2$ and then prove that any cycle in $G$ would imply an algebraic dependence amongst the radii of the original sphere packing.

Next, we use the now-known structure of the contact graph to prove that every generic radii sphere packing with contact graph $G \oplus K_2$ is stress-free (\Cref{l:stressfree}).
To do so we exploit the structure of $G \oplus K_2$ when $G$ is a forest --
namely that $G \oplus K_2$ is 3-degenerate (i.e., every subgraph has minimal degree at most 3) -- to inductively construct our original sphere packing.
We then use a well-known result from rigidity theory (\Cref{l:0ext}) to prove that our inductive construction preserves the stress-free property,
and hence prove that the original sphere packing is stress-free.

Finally,
we show that for every forest $G$ that is also a penny graph,
the graph $G \oplus K_2$ is the contact graph of an open set of sphere packings.
This is done by exploiting a well-known correspondence between sphere packings and penny graphs (\Cref{t:penny}).
It follows from this that there must exist a generic radii sphere packing with contact graph $G \oplus K_2$,
concluding the proof of \Cref{t:main}.

The paper is structured as follows.
In \Cref{sec:prelim} we carefully describe the background knowledge that is required throughout the paper.
In \Cref{sec:algfun} we prove the three technical lemmas that allow us to reduce many of the difficult 3-dimensional problems encountered during the proof of \Cref{t:main} into more manageable 2-dimensional problems.
We bring all of this together in \Cref{sec:main} to prove \Cref{t:main}.
We conclude the paper with a conjecture regarding the complexity of determining when a graph is a contact graph of a sphere packing with generic radii.
Additional background results regarding algebraic functions are contained in \Cref{app:algebraic functions}.

\section{Preliminaries}\label{sec:prelim}

In this section we outline some of the important concepts we use throughout the paper.

\subsection{Sphere packings, circle packings and bar-and-joint frameworks}\label{sec:prelim-sphere}

A \emph{(bar-and-joint) framework in $\mathbb{R}^d$} is a pair $(G,p)$ where $G=(V,E)$ is a (finite simple) graph and $p:V \rightarrow \mathbb{R}^d$ is a map known as a \emph{realisation} of $G$.
An \emph{equilibrium stress} of a framework $(G,p)$ is a map $\sigma : E \rightarrow \mathbb{R}$ such that for each vertex $v \in V$ with neighbourhood $N_G(v)$, the equality $\sum_{w \in N_G(v)} \sigma_{vw}(p_v-p_w) =0$ holds.
A framework is said to be \emph{stress-free} if the only equilibrium stress is the zero map.
The following necessary combinatorial condition for a framework to be stress-free is the $d$-dimensional analogue of the original observation of Maxwell \cite{maxwell};
see \cite[Theorem 2.5.4]{gss} for a detailed proof of the result.

\begin{theorem}\label{t:maxwell}
	Let $(G,p)$ be a stress-free framework in $\mathbb{R}^d$.
	If the graph $G=(V,E)$ has at least $d$ vertices then $|E| \leq d|V| - \binom{d+1}{2}$.
\end{theorem}

Fix $\mathbb{S}^{d-1}$ to be the unit sphere of $\mathbb{R}^d$,
i.e., the set of all points $x = (x_1,\ldots,x_d) \in \mathbb{R}^d$ where $\|x\| := \sqrt{\sum_{i=1}^d x_i^2} =1$.
A \emph{$d$-dimensional sphere packing} is a triple $P=(G,p,r)$ where $G=(V,E)$ is a (finite simple) graph,
and $p: V \rightarrow \mathbb{R}^d$ and $r:V \rightarrow \mathbb{R}_{>0}$ are maps such that the inequality $\|p_v-p_w\| \geq r_v +r_w$ holds for every distinct pair of vertices $v,w \in V$, with equality if and only if $vw \in E$.
If the set $\{r_v :v \in V\}$ is algebraically independent then $P$ is said to have \emph{generic radii}.
The framework $(G,p)$ is called the \emph{associated framework} of the $d$-dimensional sphere packing $P$.
A $d$-dimensional sphere sphere packing is then said to be \emph{stress-free} if its associated framework is stress-free.
We simplify the terminology of ``$d$-dimensional sphere packing'' to \emph{sphere packing} if $d=3$ and \emph{circle packing} if $d=2$.

Given a graph $G=(V,E)$ and a positive integer $d$,
we define $\radii_d(G)$ (also denoted $\radii(G)$ when the value of $d$ is clear from context) to be the set of all vectors $r \in \mathbb{R}^V_{>0}$ for which there exists some $d$-dimensional sphere packing $P=(G,p,r)$.
As $\radii_d(G)$ is a semi-algebraic set defined over $\mathbb{Q}$,
it has an open interior if and only if it contains a generic point (i.e., a vector with algebraically independent coordinates).
These equivalent properties are also linked to the existence of stress-free packings by the following result,
which can be proven by trivial adjustments to the proof of \cite[Proposition 5.3]{congortheran2019}.

\begin{proposition}\label{p:openradii}
	Suppose that there exists a stress-free $d$-dimensional sphere packing with contact graph $G$.
	Then the set $\radii_d(G)$ has non-empty interior and there exists a generic radii $d$-dimensional sphere packing with contact graph $G$.
\end{proposition}

Every 1-dimensional sphere packing is stress-free;
this follows from the observation that their contact graphs are always linear forest.
This is not true in general for higher dimensional sphere packings;
for example, if $P = (G,p,r)$ is a circle packing and $G=(V,E)$ is a triangulation (i.e., $|E|=3|V|-6$) with $|V| \geq 4$,
then $P$ is not stress-free by \Cref{t:maxwell}.
Connelly, Gortler and Theran proved that every generic radii circle packing is stress-free,
the 2-dimensional analogue of \Cref{conj:main} and the converse of \Cref{p:openradii}.

\begin{theorem}[\cite{congortheran2019}]\label{t:genconj2}
	Let $G=(V,E)$ be the contact graph of a circle packing $P$ with generic radii.
	Then $P$ is stress-free,
	and hence $|E| \leq 2|V|-3$ if $|V| \geq 2$.
\end{theorem}

The family of sphere packings we are most interested in are those with contact graph $G \oplus K_2$.
This is, in part, because of the following observation of Kirkpatrick and Rote;
see \cite[Proposition 4.5]{HK01} for a sketch of their proof.

\begin{theorem}\label{t:penny}
	A graph $G \oplus K_2$ is the contact graph of a sphere packing if and only if $G$ is a penny graph.
\end{theorem}

It is known that a penny graph with $n$ vertices has at most $\lfloor 3n - \sqrt{12n-3} \rfloor$ edges,
and this bound is tight \cite{Har}.
With this we obtain a tight upper bound for the number of contacts for any sphere packing with contact graph $G \oplus K_2$.

\begin{corollary}\label{c:pennysphere}
	Let $P$ be a sphere packing with contact graph $G \oplus K_2$ for some graph $G=(V,E)$.
	Then $P$ has at most $\lfloor 5|V| +1 - \sqrt{12|V|-3} \rfloor$ contacts,
	and this bound is tight.
\end{corollary}

As alluded to in the introduction, \Cref{c:pennysphere} informs us that we cannot directly use the same methods as Connelly, Gortler and Theran \cite{congortheran2019} when proving \Cref{t:main}.
To utilise their method for a given sphere packing contact graph $G$ with $n$ vertices and $m$ edges,
it is required that the set of sphere packings with contact graph $G$ is a smooth $(4n-m)$-dimensional manifold.
However \Cref{c:pennysphere} informs us that for each $n \in \mathbb{N}$ there exists a sphere packing contact graph with $n$ vertices and 
\begin{equation*}
	m = \left\lfloor 5(n-2) +1 - \sqrt{12(n-2)-3} \right\rfloor = \left\lfloor 5n - 9 - \sqrt{12n-27} \right\rfloor
\end{equation*}
edges.
The set of sphere packings with such a contact graph cannot be a $(4n-m)$-dimensional manifold when $n$ is large, since $4n-m = \lfloor -n + 9 + \sqrt{12n-27} \rfloor < 0$.

\subsection{3-dimensional M\"{o}bius transforms}

A \emph{(3-dimensional) M\"{o}bius transform} is a rational map $\phi: \mathbb{R}^3 \dashrightarrow \mathbb{R}^3$ where there exists points $a,b \in \mathbb{R}^3$, scalars $\lambda \neq 0$, $\tau \in \{0,2\}$, and linear isometry $A: \mathbb{R}^3 \rightarrow \mathbb{R}^3$ such that for each $u \in \mathbb{R}^3 \setminus \{b\}$ we have 
\begin{equation*}
	\phi(u) := a + \frac{\lambda A(u-b)}{\|u-b\|^\tau} .
\end{equation*}
These transforms are of interest since one of three things can happen when one is applied to a sphere $S$.
\begin{enumerate}
	\item Either the point $b$ is not contained in the convex hull of $S$, or $\tau = 0$: in this case $\phi(S)$ is a sphere and any point that was inside/outside the convex hull of $S$ remains so.
	\item The point $b$ is contained in the interior of the convex hull of $S$ and $\tau = 2$: in this case $\phi(S)$ is a sphere and any point that was inside (respectively, outside) the convex hull of $S$ is now outside (respectively, inside) the convex hull of $S$.
	\item The point $b$ is contained in $S$ and $\tau = 2$: in this case $\phi(S) = H\setminus\{b\}$, where $H$ is an affine hyperplane parallel to the tangent plane of $S$ at $b$.
\end{enumerate}
We define two sphere packings $P$ and $P'$ of $n$ spheres to be \emph{M\"{o}bius-equivalent} if there exists a M\"{o}bius transform that maps $P$ to $P'$
Note that if $n \geq 2$,
any such M\"{o}bius transform must satisfy condition (i) for every sphere in $P$ so that every element of $P'$ is a sphere and no sphere of $P'$ contains any other sphere of $P'$ within its interior.
Since M\"{o}bius transforms are injective, any pair of M\"{o}bius-equivalent sphere packings must have identical contact graphs.

\subsection{Algebraic functions}

Let $D \subset \mathbb{R}^n$ and $D' \subset \mathbb{R}$ be open subsets and let $\mathbb{F}$ be a subfield of $\mathbb{R}$.
A \emph{(real) algebraic function over $\mathbb{F}$} (from $D$ to $D'$) is a continuous map $f:D \rightarrow D'$ where there exists a polynomial $p \in \mathbb{F}[Y,X_1,\ldots,X_n]$ such that $p \neq 0$ and $p(f(x),x_1,\ldots,x_n) = 0$ for every $x = (x_1,\ldots,x_n) \in D$.
Equivalently,
$f$ is algebraic if and only if there exists $n$-variable polynomials $p_0,\ldots,p_m$ over $\mathbb{F}$ such that $p_k \neq 0$ for each $k\in \{0,\ldots,m\}$ and $\sum_{k=0}^m p_k(x) f(x)^k = 0$ for every $x \in D$.
Supposing that $0 \in D'$,
the zero set $Z(f) \subset D$ of a non-constant algebraic function $f$ is always contained within a proper Zariski closed subset of $\mathbb{R}^n$ intersected with $D$;
this follows from two observations:
(i) if $Z(p) \subset \mathbb{R}^n$ is the zero set of $p$ (where $p(f(x),x)=0$ for each $x \in D$) then $\{0\} \times Z(f) \subset Z(p) \cap (\{0\} \times D)$,
and (ii) $Z(p) \cap (\{0\} \times D) = \{0\} \times D$ if and only if $Z(p) \supset \{0\} \times \mathbb{R}^n$
which can only occur if $f$ is constant.
Unless stated otherwise, we assume that $D'= \mathbb{R}$.

A more detailed analysis of algebraic functions is given in \Cref{app:algebraic functions}.
The main non-standard result regarding algebraic functions that we require is the following.

\begin{lemma}\label{l:implicit}
	Let $\mathbb{F} \subset \mathbb{R}$ be a subfield, let $D \subset \mathbb{R}^n$ be an open set and let $x,y : D \rightarrow \mathbb{R}$ be continuous functions.
	Suppose there exist algebraic functions $p,q,\alpha,\beta \in \mathbb{A}_D[X_{n_1},\ldots,X_{n_k};\mathbb{F}]$ where
	\begin{align}
		\label{line1} x(t)^2 + y(t)^2 &= p(t),\\
		\label{line2} (x(t)-\alpha(t))^2 + (y(t)-\beta(t))^2 &= q(t)
	\end{align}
	for all $t \in D$.
	If at least one of $\alpha,\beta$ is not the zero map then $x,y \in \mathbb{A}_D[X_{n_1},\ldots,X_{n_k};\mathbb{F}]$.
\end{lemma}

\begin{proof}
	By rearranging \cref{line2} and squaring both sides we see that
	\begin{equation*}
		4 \alpha(t)^2 x(t)^2 = \left( x(t)^2 + \alpha(t)^2 + (y(t)-\beta(t))^2 - q(t)   \right)^2.
	\end{equation*}
	By combining the above equation and \cref{line1} we see that the following equation holds for all $t \in D$:
	\begin{align}\label{line3} 
	\left( 2\beta(t)y(t) - \left( p(t) - q(t) + \alpha(t)^2 + \beta(t)^2\right) \right)^2 +  4\alpha(t)^2y(t)^2 - 4 \alpha(t)^2p(t) = 0.
	\end{align}
	When \cref{line3} is expanded out we obtain an equation $a_0(t) + a_1(t)y(t) + a_2(t)y(t)^2 = 0$ for some algebraic functions $a_0,a_1,a_2:D \rightarrow \mathbb{R}$ over $\mathbb{F}$ with $a_2(t) = 4(\alpha(t)^2 + \beta(t)^2)$.
	Since at least one of $\alpha,\beta$ is not the zero map,
	$a_2$ is not the zero map.
	Hence $y$ is algebraic over $\mathbb{F}$ (\Cref{l:algcoeff}).
	Since the sum of any two algebraic functions over $\mathbb{F}$ is an algebraic function over $\mathbb{F}$,
	it now follows from \cref{line1} that $x$ is also algebraic over $\mathbb{F}$.
	As at least one of $\alpha,\beta$ is not the zero map,
	there exists a dense set of variables $s \in D$ where \cref{line1,line2} have a finite set of solutions over the set $\{t \in D : t_{n_i} = s_{n_i} \text{ for each $i \in \{1,\ldots,k\}$} \}$.
	Since both $x$ and $y$ are continuous, it now follows that $x,y \in \mathbb{A}_D[X_{n_1},\ldots,X_{n_k};\mathbb{F}]$.	
\end{proof}

\section{Flattening sphere packings}\label{sec:algfun}

In this section we provide multiple important technical lemmas that allow us to consider our difficult 3-dimensional problem as a simpler 2-dimensional problem.
This will be split into two parts:
(i) obtaining a standard form for our sphere packings via M\"{o}bius transforms that, for the most part, preserves generic radii, and will allow us to reduce the problem from a 3-dimensional one to a 2-dimensional one (\Cref{l:correctradii});
(ii) proving that any results obtained for the standard form can be carried back to the original setting of generic radii (\Cref{l:openset}).

\subsection{Obtaining a standard form}\label{subsec:sf}

Given a sphere packing $P=(G \oplus K_2,p,r)$ ($K_2$ being the complete graph with vertices $\{a,b\}$) with generic radii $r$,
we now show the existence of a M\"{o}bius-equivalent sphere packing $P'=(G \oplus K_2,p',r')$ where $r_a=r_b = 1$ and the restriction of $r'$ to the vertices of $G$ forms an algebraically independent set.
If we move the packing so that $p'=(0,0,-1)$ and $p'_b = (0,0,1)$,
our choice of standard form forces the points $p'_v$ to all lie in the $xy$-plane.
This positioning is desirable for later,
as it effectively reduces the problem to a 2-dimensional one.
\Cref{fig: standard form} depicts the standard form that we desire.

\begin{figure}[ht]
    \centering
    \includegraphics[width = 0.3\textwidth]{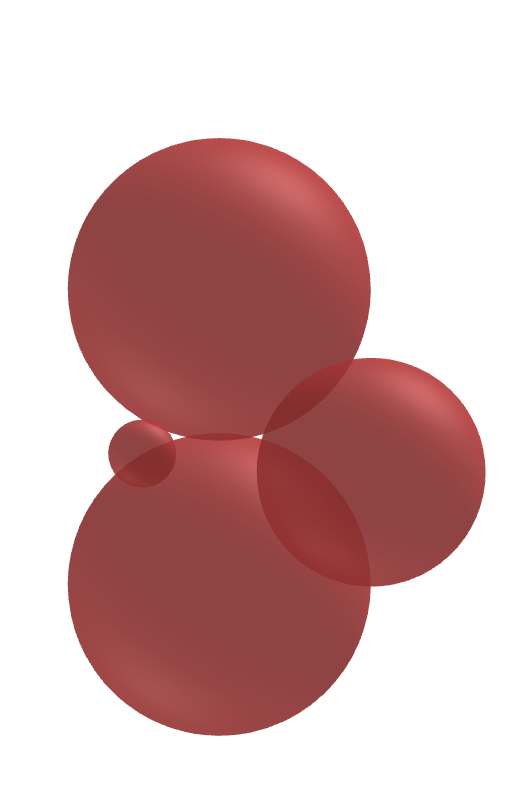}
    \caption{The standard form sphere packing, with two spheres of identical size (the top and bottom spheres) in contact with all other spheres.}
    \label{fig: standard form}
\end{figure}

We begin with the following result.

\begin{lemma}\label{l:thefix}
	Let $G=(V,E)$ be a graph, let $K_2$ be the complete graph with vertices $\{a,b\}$ and let $P = (G \oplus K_2,p,r)$ be a sphere packing with generic radii.
	Then there exists a M\"{o}bius-equivalent sphere packing $P' = (G \oplus K_2,p',r')$ where $r_a' > r_v'$ and $r_b' > r_v'$ for all $v \in V$.
\end{lemma}

\begin{proof}
	By switching the vertices $a,b$ if necessary,
	we may suppose without loss of generality that $r_a > r_b$.
	By rotating and translating $P$,
	we may also assume that $p_a = (0,0,-r_a)$ and $p_b = (0,0,r_b)$.
	For each $t \neq 0$,
	define $\phi^t$ to be the M\"{o}bius transform where
	\begin{equation*}
		\phi^t(u) = (t,0,0) + \frac{t^2 (u-(t,0,0))}{\|u-(t,0,0)\|^2} .
	\end{equation*}
	Each M\"{o}bius transform $\phi^t$ fixes the points $(0,0,0)$ and $(t,0,0)$,
	and has singularity $(t,0,0)$.
	Fix $e_x = (1,0,0)$ and choose any vertex $v$ of $G \oplus K_2$ and any non-zero scalar $t$ such that $\|te_x - p_v\| > r_v$.
	The intersection of $S_v$ and the line through $te_x$ are the two points
	\begin{equation*}
		a_v = p_v + \frac{r_v(t e_x - p_v)}{\|te_x - p_v\|} , \qquad b_v = p_v - \frac{r_v(t e_x - p_v)}{\|te_x - p_v\|} .
	\end{equation*}
	From this,
	we compute
	\begin{equation*}
		\phi^t(a_v) = te_x + \frac{t^2(t e_x - p_v)}{\|te_x - p_v\|(r_v - \|te_x - p_v\|)},  \qquad \phi^t(b_v) = te_x - \frac{t^2(t e_x - p_v)}{\|te_x - p_v\|(r_v + \|te_x - p_v\|)} .
	\end{equation*}
	It now follows that the radius of the sphere $\phi^t(S_v)$ is
	\begin{align*}
		\frac{\|\phi^t(a_v) - \phi^t(b_v)\|}{2} =
		\begin{cases}
			\frac{r_v t^2}{\|te_x - p_v\|^2-r_v^2} &\text{if } \|te_x - p_v\| > r_v,\\
			\frac{r_v t^2}{r_v^2 - \|te_x - p_v\|^2} &\text{if } \|te_x - p_v\| < r_v.
		\end{cases}
	\end{align*}
	We observe two facts here:
	(i) there exists $\delta_v>0$ such that if $0 < |t| <\delta_v$ then the first value is taken,
	and (ii) as $t \rightarrow \pm \infty$, the radii converges to $r_v$.
	
	Define the following two rational maps:
	\begin{equation*}
		f_v : \mathbb{R} \dashrightarrow \mathbb{R}, ~ t \mapsto \frac{r_v t^2}{\|te_x - p_v\|^2-r_v^2} , \qquad 
		f : \mathbb{R} \dashrightarrow \mathbb{R}^{V \cup \{a,b\}}, ~ t \mapsto (f_v(t))_{v \in V}.
	\end{equation*}
	By the Tarski-Seidenberg theorem, the image of $f$ (with domain restricted to semi-algebraic set of points where it is defined) is a semi-algebraic set.
	Furthermore, the image of $f$ is 1-dimensional and contains $r$ in its closure.
	Fix $C$ to be the closure of the image of $f$ under the Zariski topology on $\mathbb{R}^{V \cup \{a,b\}}$.
	We observe the following two facts:
	(i) 	$C$ is both irreducible and 1-dimensional (this can be seen by extending the domain and codomain of $f$ to projective spaces);
	(ii) since $C$ contains a point with algebraically independent coordinates (namely $r$),
	it is not contained in any rationally-defined algebraic subset of $\mathbb{R}^{V \cup \{a,b\}}$.
	An immediate consequence of these two facts is that $C$ contains a dense subset of points with have algebraically independent coordinates.
	Hence there exists a dense subset $D \subset \mathbb{R}$ where for each $t \in D$,
	the coordinates of $f(t)$ are algebraically independent.
	
	Fix $\delta := \min\{ \delta_v : v \in V\}$,
	with $\delta_v$ as defined above.
	If $0 < |t| < \delta$ then the radii of $\phi^t(P)$ is exactly $f(t)$.
	For each $t>0$,
	fix $S(t)$ to be the sphere with radius $t$ and centre $(t,0,0)$.
	Then there exists $T>0$ such that for all $0<t<T$,
	the sphere $S(t)$ does not overlap with, nor is contained in, any sphere $S_v$ with $v \in V$.
	Choose $0 < \varepsilon < \min\{\delta,T, r_a,r_b\}$ such that $\varepsilon \in D$.
	It is immediate that the sphere packing $\phi^\varepsilon(P)$ has generic radii $f(\varepsilon)$.
	Every sphere $S_v$ with $v \in V$ is mapped to a sphere contained in $S(\varepsilon)$, and so each sphere $\phi^\varepsilon(S_v)$ has radius less than $\varepsilon$.
	Using our prior formula,
	the spheres $\phi^\varepsilon(S_a)$ and $\phi^\varepsilon(S_b)$ have radii $r_a$ and $r_b$ respectively.
	We conclude the result by setting $P' = \phi^\varepsilon(P)$.
\end{proof}

\begin{lemma}\label{l:correctradii}
	Let $G=(V,E)$ be a graph, let $K_2$ be the complete graph with vertices $\{a,b\}$ and let $P = (G \oplus K_2,p,r)$ be a sphere packing with generic radii.
	Then there exists a M\"{o}bius-equivalent sphere packing $P' = (G \oplus K_2,p',r')$ where the following holds:
	\begin{enumerate}
		\item $p'_v=(x_v,y_v,0)$ for each $v \in V$,
		\item $p'_a = (0,0,-1)$, $p'_b = (0,0,1)$ and $r'_a = r'_b =1$, and
		\item the set $\{r'_v : v \in V \}$ is algebraically independent.
	\end{enumerate}	
\end{lemma}

\begin{proof}	
	By applying \Cref{l:thefix} and switching the vertices $a,b$ if necessary,
	we may suppose without loss of generality that $r_a > r_b > r_v$ for all $v \in V$.
	By rotating and translating $P$,
	we may also assume that $p_a = (0,0,-r_a)$ and $p_b = (0,0,r_b)$.
	Define the sphere packing $\tilde{P} = (G,\tilde{p},\tilde{r})$ by setting $\tilde{p}_v=p_v/r_b$ and $\tilde{r}_v=r_v/r_b$ for each $v \in V$.
	Since the set $\{r_v : v \in V \cup \{a,b\}\}$ is algebraically independent,
	the set $\{\tilde{r}_v : v \in V \cup \{a\}\}$ is also algebraically independent.
	Our proof now proceeds in two steps.
	We first find the unique M\"{o}bius transform $\phi$ that maps $\tilde{P}$ to a sphere packing $P'= (G \oplus K_2,p',r')$ where $x'_{a}=x'_b=y'_{a}=y'_b=0$, $z'_{a} = -1$, $z'_b= 1$ and $r'_a = r'_b = 1$.
	After constructing $P'$ and its corresponding M\"{o}bius transform $\phi$,
	we then prove that the set $\{r'_v : v \in V\}$ is algebraically independent.
	
	Define the M\"{o}bius transforms $\phi_1,\phi_2,\phi_3,\phi_4 : \mathbb{R}^3 \dashrightarrow \mathbb{R}^3$ where for each point $u = (u_x,u_y,u_z)$ in $\mathbb{R}^3$ we have
	\begin{align*}
		\phi_1(u) &:= u - \left(0,0,\frac{4\tilde{r}_a}{\tilde{r}_a-1} \right),\\
		\phi_2(u) &:= \frac{u}{\|u\|^2} \qquad \text{(defined for $u\neq (0,0,0)$)},\\
		\phi_3(u) &:= \frac{-4\tilde{r}_a(\tilde{r}_a+1)}{(\tilde{r}_a-1)^2} u,\\
		\phi_4(u) &:= u - \left(0,0,\frac{\tilde{r}_a+1}{\tilde{r}_a-1} \right).
	\end{align*}
	Now define the M\"{o}bius transform $\phi:= \phi_4 \circ \phi_3 \circ \phi_2 \circ \phi_1$,
	which is defined at every point except the singularity $\sigma := \left(0,0,4\tilde{r}_a/(\tilde{r}_a-1) \right)$.
	Note that
	\begin{equation*}
		\phi(0,0,0) = (0,0,0), \qquad \phi(0,0,2) = (0,0,2), \qquad \phi(0,0,-2\tilde{r}_a) = (0,0,-2).
	\end{equation*}
	Since $\tilde{r}_a = r_a/r_b >1$,
	the point $\sigma$ is not contained inside any spheres of $\tilde{P}$:
	$\sigma$ is not contained in the spheres for $a$ or $b$ since $\sigma >4$,
	and $\sigma$ is not contained in any other sphere since $\tilde{r}_v < 1$ for all $v \in V$ (and hence, taking into account that the spheres in $\tilde{P}$ for $a,v$ are in contact, we have $\tilde{z}_v + \tilde{r}_v < 2\tilde{r}_v < 4 < \sigma$).
	Hence $\phi(\tilde{P})$ is a sphere packing with contact graph $G\oplus K_2$.
	We now set $P' = \phi(\tilde{P})$.
	
	We now must prove that the set $\{r'_v : v \in V \}$ is algebraically independent.
	Define the sphere packing $Q := \phi_2 \circ \phi_1 (\tilde{P})$ with contact graph $G \oplus K_2$,
	centres $q: V\cup \{a,b\} \rightarrow \mathbb{R}^3$ and radii $s : V\cup \{a,b\} \rightarrow \mathbb{R}_{>0}$.
	Note that for every $v \in V \cup \{a,b\}$ we have
	\begin{equation*}
		r'_v = \frac{4\tilde{r}_a(\tilde{r}_a+1)}{(\tilde{r}_a-1)^2} s_v 
	\end{equation*}
	Hence the set $\{r'_v : v \in V\}$ is algebraically independent if the set $\{s_v : v \in V \}$ is algebraically independent over the field $\mathbb{F} := \mathbb{Q}(\tilde{r}_a)$.
	
	For $t>0$,
	fix $S_t$ to be any sphere of radius $t$ that is in contact with the spheres $S_a,S_b$ corresponding to vertices $a,b$ in $\tilde{P}$ (here we do not care whether the sphere $S_t$ intersects with any of the other spheres in $\tilde{P}$).
	Whilst there is a 1-dimensional family of spheres that $S_t$ can be,
	there exists a unique scalar $\alpha(t)>0$ such that each set $\phi_2 \circ \phi_1(S_t)$ is a sphere with radius $\alpha(t)$ which is in contact with the spheres $\phi_2\circ \phi_1(S_a)$ and $\phi_2\circ \phi_1(S_b)$ corresponding to vertices $a,b$ in $Q$ (yet again, we do not care whether it intersects with any of the other spheres in $Q$).
	Consider the map $\alpha: \mathbb{R}_{>0} \rightarrow \mathbb{R}_{>0}$ which maps each $t$ to its unique value $\alpha(t)$.
	Since M\"{o}bius transforms are injective and continuous except at their singularity,
	the map $\alpha$ is injective and continuous.
	Note that for each $v \in V$ we have $\alpha(\tilde{r}_v) = s_v$.
	It is now sufficient to prove $\alpha$ is a bijective algebraic function over $\mathbb{F}$,
	as it then follows from \Cref{l:algfunalgdep,l:algfuninv} that the set $\{s_v: v \in V \}$ is algebraically independent over $\mathbb{F}$.
	
	Similarly to before, fix $T_t$ to be any sphere of radius $t>0$ that is in contact with the spheres $T_a,T_b$ corresponding to vertices $a,b$ in $Q$, and fix to be the unique scalar $\beta(t)>0$ such that each set $(\phi_2 \circ \phi_1)^{-1}(T_t)$ is a sphere with radius $\beta(t)$ in contact with the spheres $(\phi_2\circ \phi_1)^{-1}(T_a)$ and $(\phi_2\circ \phi_1)^{-1}(T_b)$ in $\tilde{P}$.
	Then the map $\beta: \mathbb{R}_{>0} \rightarrow \mathbb{R}_{>0}$ which maps each $t$ to its unique value $\beta(t)$ is the inverse of $\alpha$.
	Hence $\alpha$ is bijective.
	
	Now fix each sphere $S_t$ to be the unique sphere of radius $t$ in contact with $S_a,S_b$ with radius $t$ and centre $p(t) = (x(t),0,z(t))$ for some $x(t)>0$ and $z(t)$.
	By the implicit function theorem (see \cite[Theorem 9.28]{rudin}),
	both $x$ and $z$ are continuous.
	By \Cref{l:implicit}, both $x,z-1$ (and hence $x,z$) are algebraic functions over $\mathbb{F}$ as they satisfy the equations
	\begin{align*}
		x(t)^2 + (z(t) - 1)^2 &= (t+1)^2,\\
		x(t)^2 + (z(t) - 1 + (1-\tilde{r}_a))^2 &= (t+\tilde{r}_a)^2
	\end{align*}
	for all $t>0$, and $\tilde{r}_a > 1$.
	Fix the non-zero scalar $\gamma:= \frac{4 \tilde{r}_a}{\tilde{r}_a -1} \in \mathbb{F}$.
	Note that the centre of the sphere $\phi_2 \circ \phi_1(S_t)$ is the point $q(t) = (\lambda (t), 0,\mu(t))$,
	where $\lambda,\mu$ are algebraic functions over $\mathbb{F}$ defined at each $t>0$ to be
	\begin{equation*}
		\lambda(t) := \frac{x(t)}{x(t)^2 + (z(t) - \gamma)^2}, \qquad \mu(t) := \frac{z(t)-\gamma}{x(t)^2 + (z(t)- \gamma)^2}.
	\end{equation*}
	Since $q_b= (0,0,\frac{\tilde{r}_a-1}{\tilde{r}_a+1})$ and $s_b = \gamma^{-1}$,
	it follows that $\alpha(t)$ satisfies the equation
	\begin{equation*}
		\lambda(t)^2 + \left( \mu(t)-\frac{\tilde{r}_a-1}{\tilde{r}_a+1} \right) ^2 = \left( \alpha(t) +  \gamma^{-1} \right)^2.
	\end{equation*}
	As $\lambda,\mu$ are algebraic functions over $\mathbb{F}$ and $\tilde{r}_a,\gamma \in \mathbb{F}$,
	the map	$\alpha$ is an algebraic function over $\mathbb{F}$ (\Cref{l:algcoeff}).
\end{proof}

\subsection{Properties of radii sets}

We require the following notation.
For a fixed graph $G=(V,E)$ and a copy of $K_2$ with vertices $\{a,b\}$,
we define the following sets:
\begin{align*}
	\radii_{a=b=1}(G \oplus K_2) &:= \left\{(r_v)_{v \in V}: (r_v)_{v \in V \cup \{a,b\}} \in \radii(G \oplus K_2) ,~ r_a=r_b=1 \right\},\\		
	\radii_{a \neq b}(G \oplus K_2) &:= \left\{r \in \radii(G \oplus K_2): r_a \neq r_b \right\},\\
	\radii_{a>b}(G \oplus K_2) &:= \left\{r \in \radii(G \oplus K_2): r_a>r_b \right\},\\
	\radii_{a>b>V}(G \oplus K_2) &:= \left\{r \in \radii(G \oplus K_2): r_a>r_b > r_v \text{ for all } v \in V \right\}.
\end{align*}
Observe that the latter three sets are all open subsets of $\radii (G \oplus K_2)$.

\begin{lemma}\label{l:thefix2}
	Let $G=(V,E)$ be a graph and $K_2$ be the complete graph with vertices $\{a,b\}$.
	If the set $\radii_{a > b}(G \oplus K_2)$ has non-empty interior,
	then the set $\radii_{a>b>V}(G \oplus K_2)$ has non-empty interior.
\end{lemma}

\begin{proof}
	Fix the open set $X = (\mathbb{R}^3)^{V \cup \{a,b\}} \times \mathbb{R}_{>0}^{V \cup \{a,b\}}$ and the projection $\pi : X \mapsto \mathbb{R}^{V \cup\{a,b\}}$ where $\pi(p,r) = r$.
	Define $Z \subset X$ to be the set of pairs $(p,r)$ that each correspond to a sphere packing $(G\oplus K_2,p,r)$ with $r_a > r_b$.	
	With this set-up we see that $\pi(Z) =\radii_{a > b}(G \oplus K_2)$.
	As $\radii_{a > b}(G \oplus K_2)$ has non-empty interior,
	there exists an open set $U \subset Z$ such that $\pi(U)$ is an open subset of $\mathbb{R}^{V \cup\{a,b\}}$.
	Choose $(p',r') \in U$ so that $r'$ has algebraically independent coordinates,
	and set $P = (G \oplus K_2,p',r')$.
	By \Cref{l:thefix},
	there exists a M\"{o}bius transformation $\phi$ (with singularity $\sigma$) such that $\phi(P)$ is a sphere packing and $\pi \circ \phi(P) \in \radii_{a > b}(G \oplus K_2)$.
	
	We now require the following notation.
	Given a pair $(x,t) \in \mathbb{R}^3 \times \mathbb{R}_{>0}$ with $x \neq \sigma$,
	fix $\phi(x,t)$ to denote the pair $(x',t')$ chosen so that $\phi(t \mathbb{S}^2 + x) = t' \mathbb{S}^2 + x'$ (recall that $\mathbb{S}^2$ is the sphere with centre $(0,0,0)$ and radius $1$).
	Fix $X' \subset X$ to be the open dense subset of pairs $(p,r)$ where $p_v \neq \sigma$ for all $v \in V \cup \{a,b\}$,
	and fix $Z' = Z \cap X'$.
	We note here that both $X'$ and $Z'$ are non-empty (since they both contain $(p',r')$), and each is an open subset (of $X$ and $Z$ respectively).
	We now define $f: X' \rightarrow X$ to be the continuous map where for each $(p,r)$ with $p_v \neq \sigma$ for all $v \in V \cup \{a,b\}$,
	$f(p,r) = (\phi(p_v,r_v))_{v \in V \cup \{a,b\}}$.
	Since the map $(x,t) \mapsto \phi(x,t)$ is an open map (when restricted to its open dense set of well-defined points in its domain),
	so too is the map $f$.
	Since $\phi$ (with a restricted domain) is invertible,
	it follows that the domain and codomain restricted map $h = f|_{Z'}^{Z}$ is an open map.
	Hence the composition map $g =\pi \circ h$ is an open map.
	Furthermore, the image of $g$ is exactly the set $\radii_{a > b}(G \oplus K_2)$.
	We now note three important facts: (i) $g$ is a continuous open map on the non-empty open subset $U \cap Z'$ of $Z'$; (ii) $\radii_{a > b > v}(G \oplus K_2)$ is an open subset of $\radii_{a > b}(G \oplus K_2)$; (iii) $g(p',r') \in \radii_{a > b}(G \oplus K_2)$.
	Hence,
	$g(U \cap Z') \cap \radii_{a > b > v}(G \oplus K_2)$ is a non-empty open subset of $\mathbb{R}^{V \cup\{a,b\}}$,
	which concludes the proof.
\end{proof}

\begin{lemma}\label{l:openset}
	Let $G=(V,E)$ be a graph and $K_2$ be the complete graph with vertices $\{a,b\}$.
	If the set $\radii_{a=b=1}(G \oplus K_2)$ has non-empty interior,
	then the set $\radii(G \oplus K_2)$ has non-empty interior.
\end{lemma}

\begin{proof}
	As switching $a$ and $b$ is an automorphism of $G \oplus K_2$,
	the set $\radii_{a \neq b}(G \oplus K_2)$ has non-empty interior if and only if the set $\radii_{a>b}(G \oplus K_2)$ has non-empty interior.
	Furthermore,
	as set $\radii_{a \neq b}(G \oplus K_2)$ is the intersection of an open dense set and $\radii(G \oplus K_2)$,
	it follows that the set $\radii(G \oplus K_2)$ has non-empty interior if and only if the set $\radii_{a \neq b}(G \oplus K_2)$ has non-empty interior.
	Hence the set $\radii(G \oplus K_2)$ has non-empty interior if and only if the set $\radii_{a>b}(G \oplus K_2)$ has non-empty interior.
	By \Cref{l:thefix2},
	the set $\radii_{a>b}(G \oplus K_2)$ has non-empty interior if and only if the set $\radii_{a>b>V}(G \oplus K_2)$ has non-empty interior.
	Hence the set $\radii(G \oplus K_2)$ has non-empty interior if and only if the set $\radii_{a>b>V}(G \oplus K_2)$ has non-empty interior.
	
	Define $Z \subset (\mathbb{R}^3)^{V \cup \{a,b\}} \times \mathbb{R}_{>0}^{V \cup \{a,b\}}$ be the set of pairs $(p,r)$ corresponding to a sphere packing $(G\oplus K_2,p,r)$ where $p_a = (0,0,r_a)$ and $p_b=(0,0,-r_b)$.	
	Define $Z' \subset Z$ to be the subset of pairs $(p,r)$ where $r_a>r_b > r_v$ for all $v \in V$ and define $Z'' \subset Z$ to be the subset of pairs $(p,r)$ where $r_a =r_b = 1$.
	Also define the projection maps
	\begin{align*}
		\pi_1 &: (\mathbb{R}^3)^{V \cup \{a,b\}} \times \mathbb{R}_{>0}^{V \cup \{a,b\}} \rightarrow \mathbb{R}_{>0}^{V \cup \{a,b\}}, ~ (p,r) \mapsto r \\
		\pi_2 &: \mathbb{R}_{>0}^{V \cup \{a,b\}} \rightarrow \mathbb{R}_{>0}^{V}, ~ (r_v)_{v \in V \cup \{a,b\}} \mapsto (r_v)_{v \in V}.
	\end{align*}
	With this we have
	\begin{equation*}
		\pi_1(Z) = \radii(G \oplus K_2), ~ ~ \pi_1(Z') = \radii_{a > b > V}(G \oplus K_2), ~ ~ \pi_2 \circ \pi_1(Z'') = \radii_{a=b=1}(G \oplus K_2).
	\end{equation*}
	
	For each pair of scalars $\lambda>\mu >0$,
	define the M\"{o}bius transforms $\phi_0^{\lambda,\mu},\phi_1^{\lambda,\mu},\phi_2^{\lambda,\mu},\phi_3^{\lambda,\mu},\phi_4^{\lambda,\mu} : \mathbb{R}^3 \dashrightarrow \mathbb{R}^3$ where for each $u = (u_x,u_y,u_z) \in \mathbb{R}^3$ we have
	\begin{align*}
		\phi_0^{\lambda,\mu}(u) &:= \frac{1}{\mu} u,\\
		\phi_1^{\lambda,\mu}(u) &:= u - \left(0,0,\frac{4(\lambda/\mu)}{(\lambda/\mu)-1} \right),\\
		\phi_2^{\lambda,\mu}(u) &:= \frac{u}{\|u\|^2} \qquad \text{(if $u\neq (0,0,0)$)},\\
		\phi_3^{\lambda,\mu}(u) &:= \frac{-4(\lambda/\mu)((\lambda/\mu)+1)}{((\lambda/\mu)-1)^2} u,\\
		\phi_4^{\lambda,\mu}(u) &:= u - \left(0,0,\frac{(\lambda/\mu)+1}{(\lambda/\mu)-1} \right).
	\end{align*}
	Define also the M\"{o}bius transform $\phi^{\lambda,\mu} := \phi_4^{\lambda,\mu} \circ \phi_3^{\lambda,\mu} \circ \phi_2^{\lambda,\mu} \circ \phi_1^{\lambda,\mu} \circ \phi_0^{\lambda,\mu}$.
	Using the techniques outlined in the proof of \Cref{l:correctradii},
	we note that
	\begin{equation*}
		\phi^{\lambda,\mu}(0,0,0) = (0,0,0), \qquad \phi^{\lambda,\mu}(0,0,2\mu) = (0,0,2), \qquad \phi^{\lambda,\mu}(0,0,-2\lambda) = (0,0,-2),
	\end{equation*}
	and for every sphere packing $P'=(G\oplus K_2,p',r')$ with $(p',r') \in Z'$,
	there exists a unique sphere packing $P''=(G\oplus K_2,p'',r'')$ with $(p'',r'') \in Z''$ such that $P'' = \phi^{r_a',r_b'}(P')$.
	From this, fix $\phi:Z' \rightarrow Z''$ to be the continuous map that sends each $(p',r')$ to its unique $(p'',r'')$ by the map $\phi^{r_a',r_b'}$.
	Note that the map $\phi$ is surjective:
	for any $\lambda > \mu >0$ and any $(p'',r'') \in Z''$,
	the triple $(G\oplus K_2,p',r') := (\phi^{\lambda,\mu})^{-1}(G\oplus K_2,p'',r'')$ is a sphere packing with $(p',r') \in Z'$.

	Fix a non-empty open subset $U \subset \mathbb{R}^V$ contained within the interior of $\radii_{a=b=1}(G \oplus K_2)$.
	Since each of the maps $\phi$, $\pi_1$ and $\pi_2$ are continuous and surjective,
	the set
	\begin{equation*}
		U' := \phi^{-1}\left( (\pi_2 \circ \pi_1)^{-1}(U) \cap Z'' \right) \neq \emptyset
	\end{equation*}
	is an open subset of $Z'$.
	Since linear projections are open maps and $\pi_1(Z') = \radii_{a>b}(G \oplus K_2)$,
	the set $\pi_1(U') \subset \mathbb{R}^{V \cup \{a,b\}}$ is open and contained within $\radii_{a>b}(G \oplus K_2)$.
	Hence $\radii_{a>b}(G \oplus K_2)$ has non-empty interior,
	and so $\radii(G \oplus K_2)$ has non-empty interior also by the aforementioned correspondence between the two sets.
\end{proof}

\section{Proof of \texorpdfstring{\Cref{t:main}}{main theorem}}\label{sec:main}

\subsection{Cycle structure of the contact graph}

We begin with the following technical result regarding a special type of framework.

\begin{lemma}\label{l:specialframework}
	Let $G=(V,E)$ be the graph formed from the cycle $(v_1,\ldots,v_k)$ (with $k \geq 3$) by adding a vertex $v_0$ adjacent to all vertices in the cycle.
	Suppose that the framework $(G,p)$ in $\mathbb{R}^2$ is constructed from the positive scalars $r_1,\ldots,r_k$ in the following way:
	\begin{enumerate}
		\item \label{item1} $p_{v_i} = (x_i,y_i)$ and $x_i^2 + y_i^2 = r_i^2 + 2r_i$ for each $1 \leq i \leq k$,
		and $x_0 = y_0 = y_1= 0$;
		\item \label{item2} the pair $p_{v_i},p_{v_{i+1}}$ are linearly independent and the vector $p_{v_{i+1}}$ is clockwise of $p_{v_i}$ for each $1 \leq i \leq k-1$,
		and similarly the pair $p_{v_1},p_{v_k}$ are linearly independent and the vector $p_{v_1}$ is clockwise of $p_{v_k}$;
		\item \label{item3} $\|p_{v_i} - p_{v_{i+1}}\| = r_i+r_{i+1}$ for each $1 \leq i \leq k-1$.
	\end{enumerate}
	(See \Cref{fig:special framework} for an example of such a framework with $k = 4$.)
	If $\|p_{v_1}-p_{v_k}\| = r_1 + r_k$ then the set $\{r_1,\ldots,r_k\}$ is not algebraically independent.
\end{lemma}

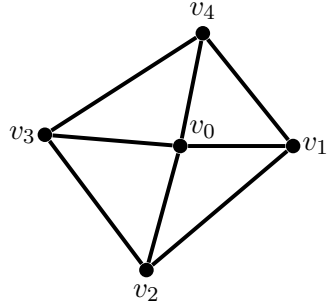
\begin{figure}[htp]
	\begin{center}
        \begin{tikzpicture}[scale=1.5]
			\node[vertex] (0) at (0,0) {};
			\node[vertex] (1) at (1,0) {};
			\node[vertex] (2) at (-0.3,-1.1) {};
			\node[vertex] (3) at (-1.2,0.1) {};
			\node[vertex] (4) at (0.2,1) {};
			
			\node (v0) at (0.2,0.15) {$v_0$};
			\node (v1) at (1.2,0) {$v_1$};
			\node (v2) at (-0.3,-1.3) {$v_2$};
			\node (v3) at (-1.4,0.1) {$v_3$};
			\node (v4) at (0.2,1.2) {$v_4$};
			
			\draw[edge] (0)edge(1);
			\draw[edge] (0)edge(2);
			\draw[edge] (0)edge(3);
			\draw[edge] (0)edge(4);
			
			\draw[edge] (1)edge(2);
			\draw[edge] (2)edge(3);
			\draw[edge] (3)edge(4);
			\draw[edge] (4)edge(1);
		\end{tikzpicture}
	\end{center}
	\caption{One example of a framework described in \Cref{l:specialframework} for $k=4$. 
	Each vertex $v_i$ is positioned at the point $p_{v_i} = (x_i,y_i)$ in $\mathbb{R}^2$.
	Every edge $v_i v_{i+1}$ with $i \in \{1,2,3\}$ has length $r_i + r_j$, and every edge $v_0 v_i$ with $i \in \{1,2,3,4\}$ has length $\sqrt{r_i^2 + 2r_i}$. 
	By \Cref{l:specialframework}, the set $\{r_1,r_2,r_3,r_4\}$ is algebraically dependent if the edge $v_1v_4$ has length $r_1+r_4$.}\label{fig:special framework}
\end{figure}

\begin{proof}
	For every positive integer $n$, fix $[n] :=\{1,\ldots,n\}$.
	For the sake of contradiction, suppose that the set $\{r_1,\ldots,r_k\}$ is algebraically independent and $\|p_{v_1}-p_{v_k}\| = r_1 + r_k$.
	Fix $D \subset \mathbb{R}^k_{>0}$ be the open convex set of vectors $s=(s_i)_{i\in[k]}$ where the construction of a realisation of $G$ from $s$ in the same way we constructed $p$ from $r$ (ignoring for now the condition that $\|p_{v_1}-p_{v_k}\|=r_1+r_k$) is possible.
	Note that this is equivalent that the following triangle inequalities hold for every $i \in [k-1]$:
	\begin{align*}
		s_i + s_{i+1} + \sqrt{s_{i+1}^2 + 2s_{i+1}}  &>  \sqrt{s_{i}^2 + 2s_{i}}, \\
		s_i + s_{i+1} + \sqrt{s_{i}^2 + 2s_{i}} &>  \sqrt{s_{i+1}^2 + 2s_{i+1}},\\
		\sqrt{s_{i}^2 + 2s_{i}} +  \sqrt{s_{i+1}^2 + 2s_{i+1}} &> s_i + s_{i+1}.
	\end{align*}
	Denote the unique realisation constructed from $s$ by $p(s) = (p_v(s))_{v \in V}$,
	where $p_{v_i}(s) = (x_i(s),y_i(s))$ for each $v_i \in V$.
	Hence the properties \ref{item1}, \ref{item2}, \ref{item3} given in the statement of \Cref{l:specialframework} hold with $r$ replaced with $s$ and $p$ replaced with $p(s)$.
	Importantly, we note that each of the maps $x_0,y_0,\ldots,x_k,y_k$ is continuous, and for all $s \in D$ we have $x_0(s)  = y_0(s) = y_1(s) =0$ and $x_1(s) = \sqrt{s_1^2 + 2s_1}$ (and so $x_1 \in \mathbb{A}_D(X_1)$).
	For every $i \in \{2,\ldots,k\}$ and every $s \in D$ we have
	\begin{align*}
		x_{i}(s)^2 + y_{i}(s)^2 &= s_{i}^2 + 2s_{i},\\
		(x_{i}(s)-x_{i-1}(s))^2 + (y_{i}(s)-y_{i-1}(s))^2 &= (s_i + s_{i-1})^2.
	\end{align*}
	Hence by inductively applying \Cref{l:implicit} we have that $x_{i},y_{i} \in \mathbb{A}_D(X_1,\ldots,X_i)$ for each $i \in [k]$.
	
	Define the algebraic function $f \in \mathbb{A}_D(X_1,\ldots,X_k)$ where
	\begin{equation*}
		f(s) := \|p_k(s) - p_1(s)\|^2 - (s_k + s_1)^2 = (x_k(s) -x_1(s))^2 + (y_k(s) - y_1(s))^2 - (s_k + s_1)^2.
	\end{equation*}
	Since $f(r) = 0$ and the set $\{r_1,\ldots,r_k\}$ is algebraically independent,
	it follows that $f = 0$.
	Hence by \Cref{l:implicit},
	$x_k,y_k \in \mathbb{A}_D(X_1,X_k)$.
	By applying the same inductive argument as before except in the reverse direction --
	in that we start with the observation that $x_k,y_k \in \mathbb{A}_D(X_1,X_k)$ and define $x_{i-1},y_{i-1}$ from $x_i,y_i$ at each step -- we see that $x_{i},y_{i} \in \mathbb{A}_D(X_1,X_i,\ldots,X_k)$ also for each $i \in [k]$ with $i \neq 1$.
	Hence $x_{i},y_{i} \in \mathbb{A}_D(X_1,X_i)$ for each $i \in [k]$.
	
	For each $\in [k]$ define the sets
	\begin{equation*}
		D_i := \left\{(s_1,s_i) : (s_\ell)_{\ell \in [k]} \in D \right\}, \qquad Z_i := \left\{s_i : (r_1,s_i) \in D_i \right\}.
	\end{equation*}
	Since $D$ is an open set containing $r$,
	each of the sets $D_i$ and $Z_i$ is an open non-empty subset of $\mathbb{R}^2$ and $\mathbb{R}$ respectively,
	and $r_i \in Z_i$.
	For each $i \in [k]$ we have $x_{i},y_{i} \in \mathbb{A}_D(X_1,X_i)$,
	and so we construct the well-defined map $\tilde{x}_i :D_i \rightarrow \mathbb{R}_{>0}$ by setting $\tilde{x}_i(s_1,s_i) := x_i(s)$ for any choice of $s = (s_\ell)_{\ell \in [k]} \in D$.
	We analogously construct the well-defined maps $\tilde{y}_i :D_i \rightarrow \mathbb{R}_{>0}$.
	Now fix $2 \leq j \leq k-1$.
	Any $s \in D$ satisfies the following equations:
	\begin{align*}
		\tilde{x}_j(s_1,s_j)^2 + \tilde{y}_j(s_1,s_j)^2 &= s_j^2 + 2s_j,\\		
		(\tilde{x}_j(s_1,s_j)-\tilde{x}_{j-1}(s_1,s_{j-1}))^2 + (\tilde{y}_j(s_1,s_j)-\tilde{y}_{j-1}(s_1,s_{j-1}))^2 &=  (s_j + s_{j-1})^2 ,\\
		(\tilde{x}_j(s_1,s_j)-\tilde{x}_{j+1}(s_1,s_{j+1}))^2 + (\tilde{y}_j(s_1,s_j)-\tilde{y}_{j+1}(s_1,s_{j+1}))^2 &=  (s_j + s_{j+1})^2 .
	\end{align*}
	By fixing $s_1=r_1$ and $s_{j \pm 1}=r_{j \pm 1}$, replacing the scalar values $\tilde{x}_{j\pm 1}(r_1,r_{j\pm 1}),\tilde{y}_{j\pm 1}(r_1,r_{j\pm 1})$ with the scalars $x_{j\pm 1},y_{j \pm 1}$ (as defined by $p$) and rearranging the equations,
	we have for every $t \in Z_j$ that
	\begin{align*}
		\tilde{x}_j(r_1,t)^2 + \tilde{y}_j(r_1,t)^2 - t^2 - 2t &= 0,\\		
		(\tilde{x}_j(r_1,t)-x_{j-1})^2 + (\tilde{y}_j(r_1,t)-y_{j-1})^2 - (t + r_{j-1})^2 &= 0,\\
		(\tilde{x}_j(r_1,t)-x_{j+1})^2 + (\tilde{y}_j(r_1,t)-y_{j+1})^2 - (t + r_{j+1})^2 &= 0.
	\end{align*}
	Hence the equations
	\begin{align}
		\label{eq:3sphere1} a^2 + b^2 - c^2 - 2c &= 0,\\
		\label{eq:3sphere2} (a - x_{j-1})^2 + (b - y_{j-1})^2 - (c + r_{j-1})^2 &= 0,\\
		\label{eq:3sphere3} (a - x_{j+1})^2 + (b - y_{j+1})^2 - (c + r_{j+1})^2 &= 0,
	\end{align}
	have infinitely many solutions $(a,b,c) \in \mathbb{R}^3$ (since $Z_j$ contains infinitely many points).
	By taking the Jacobian of \cref{eq:3sphere1,eq:3sphere2,eq:3sphere3} with respect to the variables $(a,b,c)$ and then dividing every row by 2, multiplying the right-most column by $-1$, and subtracting the top row from the bottom two rows,
	we obtain the matrix
	\begin{equation*}
		M (a,b,c):=
		\begin{bmatrix}
			a & b & c + 1 \\
			- x_{j-1} & - y_{j-1}  & r_{j-1} - 1 \\
			- x_{j+1} & - y_{j+1}  & r_{j+1} - 1
		\end{bmatrix}.
	\end{equation*}
	As \cref{eq:3sphere1,eq:3sphere2,eq:3sphere3} are polynomial equations with infinitely many solutions,
	$\det M(a,b,c) = 0$ for all $(a,b,c)\in \mathbb{R}^3$.
	Note that this occurs if and only if the vectors $v_{-} := (x_{j-1},y_{j-1},r_{j-1}-1)$ and $v_{+} := (x_{j+1},y_{j+1},r_{j+1}-1)$ are linearly dependent.
	Suppose that $v_-$ and $v_+$ are indeed linearly dependent.
	Since neither vector can have a zero as their third coordinate (as this would imply $r_{j \pm 1} =1$, contradicting that the set $\{r_1,\ldots,r_k\}$ is algebraically independent),
	there exists $\lambda \neq 0$ such that $x_{j+1} = \lambda x_{j-1}$, $y_{j+1} = \lambda y_{j-1}$ and $r_{j+1} = \lambda r_{j-1} - \lambda + 1$.
	By substituting our new values for $x_{j+1},y_{j+1},r_{j+1}$ into the equation defined in property \ref{item1} with $i=j+1$ of our constructed framework $(G,p)$,
	we see that
	\begin{equation*}
		0 = x_{j+1}^2 + y_{j+1}^2 - r_{j+1}^2 - 2r_{j+1} = \lambda^2 (x_{j-1}^2 + y_{j-1}^2 - r_{j-1}^2) + (2r_{j-1} - 1)\lambda^2 + (4 - 4r_{j-1}) \lambda -3
	\end{equation*}
	Using property \ref{item1} with $i=j-1$ of our constructed framework $(G,p)$,
	we obtain the equation
	\begin{equation*}
		(4r_{j-1} - 1)\lambda^2 + (4 - 4r_{j-1})\lambda - (4 r_{j-1} +3) = 0.
	\end{equation*}
	However this implies that $\lambda$ -- and hence also $r_{j+1}$ -- lies in the algebraic closure of $\mathbb{F}(r_{j-1})$,
	contradicting that the set $\{r_1,\ldots,r_k\}$ is algebraically independent.
	Hence the function $(a,b,c) \mapsto \det M (a,b,c)$ is a non-zero polynomial,
	contradicting our earlier claim that $\det M(a,b,c) = 0$. 
	This now concludes the proof.
\end{proof}

\begin{lemma}\label{l:main}
	Let $G=(V,E)$ and let $G \oplus K_2$ be the contact graph of a sphere packing $P$ with generic radii.
	Then $G$ is a forest.
\end{lemma}

\begin{proof}
	Suppose for contradiction that $G$ contains a cycle $C=(v_1,\ldots,v_k)$ for some $k \leq |V|-2$.	
	We first wish to reform $P$ into a more friendly form.
	By translating and rotating $P$ we may suppose that $x_{a}=x_b=y_{a}=y_b=0$, $z_{a} = -r_{a}$ and $z_b=r_b$.
	By applying \Cref{l:correctradii} we may replace $P$ with a M\"{o}bius-equivalent sphere packing where $r_{a}=1$, $r_{b} = 1$ and (in contrast to $r$ being generic) the set $\{r_v : v \in V \}$ is algebraically independent.
	Note that since every vertex is adjacent to both $a$ and $b$,
	we have for each $v \in V $ that $p_v = (x_v,y_v,0)$ for some $x_v,y_v$ where $x_v^2 + y_v^2 = r_v^2 + 2r_v$.
	Fix the set $U:= \{v_1,\ldots,v_k, a\}$ and fix $q:U \rightarrow \mathbb{R}^2$ to be the map where $q_v = (x_v,y_v)$ for each $v\in  U$.
	Define $H$ to be the subgraph of $G$ with vertex set $U$ and edge set $\{v_1 v_2, \ldots, v_{k-1} v_k, v_k v_1\} \cup \{ a v_i : 1 \leq i \leq k\}$.
	Importantly,
	$H$ contains the cycle $C$.
	By rotating $P$ we may suppose that $y_{v_1}=0$.
	However we now have a contradiction;
	$(H,q)$ is the framework described in \Cref{l:specialframework},
	hence the set $\{r_v:v \in V \}$ is not algebraically independent.
\end{proof}

\subsection{Stress-free property of the sphere packing}

Using our new knowledge of the structure of the contact graph from \Cref{l:main},
we next prove that any such sphere packings are stress-free.
To do so we need the following folk-lore result from rigidity theory.

\begin{lemma}[{\cite[Lemma 11.1.1]{w96}}]\label{l:0ext}
	Suppose $(G,p)$ and $(G',p')$ are frameworks in $\mathbb{R}^d$ where the following hold:
	\begin{enumerate}
		\item $G'=(V',E')$ is a subgraph of $G=(V,E)$ and $V = V' \cup \{v_0\}$,
		\item $p'_v = p_v$ for all $v \in V'$,
		\item $v_0$ has exactly $k \leq d$ neighbours $v_1,\ldots,v_k$ in $G$, and
		\item the points $p_{v_0}, \ldots, p_{v_k}$ are not contained in a $(k-1)$-dimensional affine subspace of $\mathbb{R}^d$.
	\end{enumerate}
	Then $(G,p)$ is stress-free if and only if $(G',p')$ is stress-free.
\end{lemma}

\begin{lemma}\label{l:stressfree}
	Let $G=(V,E)$ and let $G \oplus K_2$ be the contact graph of a sphere packing $P$ with generic radii.
	Then $P$ is stress-free.
\end{lemma}

\begin{proof}
	Fix $p:V \cup \{a,b\} \rightarrow \mathbb{R}^3$ and $r:V \cup \{a,b\} \rightarrow \mathbb{R}_{>0}$ map the vertices of $G$ to the centres and radii for their corresponding sphere respectively.
	By applying rotations and translations to $P$ we may suppose that $p_a = (0,0,-r_a)$ and $p_b = (0,0,r_b)$.
	
	We now proceed with an inductive proof.
	If $|V|=0$ then $P$ is stress-free since $p_a \neq p_b$.
	Suppose instead that the result holds for any $(k-1)$ vertex graph and fix $|V| = k$.
	By \Cref{l:main},
	$G$ is a forest.
	Hence there exists a vertex $v \in V$ with degree 1 or less in $G$,
	and so with degree 3 or less in $G \oplus K_2$.
	Fix $P'$ to be the generic radii sphere packing formed by removing the sphere for $v$ from $P$,
	and fix $p',r'$ to be its centre and radii of $P'$ respectively.
	Since $G-v$ is a forest and $Q$ has contact graph $(G-v) \oplus K_2$,
	$Q$ (and hence $((G-v) \oplus K_2,p')$) is stress-free by our inductive assumption.
	The triple $\{p_v,p_a,p_b\}$ is not colinear,
	indeed if they were then one of the spheres would be required to have radius 0 or an overlap would occur.
	If $v$ is not adjacent to any vertices in $G$ then $(G,p)$ (and hence $P$) is stress-free by \Cref{l:0ext}.
	Suppose $v$ is adjacent to a vertex $w$ in $G$.
	If the points $\{p_v,p_w,p_a,p_b\}$ are not coplanar then $(G,p)$ (and hence $P$) is stress-free by \Cref{l:0ext}.
	It now suffices to prove that the points $\{p_v,p_w,p_a,p_b\}$ are not coplanar.
	
	Suppose for contradiction that the points $\{p_v,p_w,p_a,p_b\}$ are coplanar.
	By rotating $P$ around the $z$-axis, we may suppose that $y_v=y_w=0$ also.
	Fix $H$ to be the subgraph of $G$ induced by the vertex set $U:= \{v,w,a,b\}$, and fix $q:U \rightarrow \mathbb{R}^2$ and $s:U \rightarrow \mathbb{R}_{>0}$ to be the maps where $q_u = (x_u,z_u)$ and $s_u=r_u$ for each $u\in  U$.
	Then $(H,q,s)$ is a generic radii circle packing.
	However $H$ has $2|U|-2 > 2|U|-3$ edges,
	contradicting \Cref{t:genconj2}.
	This now concludes the proof.
\end{proof}

\subsection{Proving \texorpdfstring{\Cref{t:main}}{main theorem} and corollaries}

We are now ready to prove \Cref{t:main}.

\begin{proof}[Proof of \Cref{t:main}]
	Following from \Cref{t:penny},
	we suppose throughout the proof that $G$ is a penny graph.	
	If $G \oplus K_2$ is the contact graph of a sphere packing with generic radii then $G$ is a forest by \Cref{l:main},
	and $P$ is stress-free by \Cref{l:stressfree}.
	
	Now let $G$ be a forest.
	Let $X$ be the set of all penny graph realisations of $G$ with no vertex within distance 1 of the origin,
	i.e., maps $q:V \rightarrow \mathbb{R}^2$ where $\|q_v\|>1$ for each $v \in V$ and $\|q_v-q_w\| \geq 1$ for all distinct $v,w \in V$ with equality if and only if $vw \in E$.
	The derivative of the smooth map
	\begin{equation*}
		f_G : (\mathbb{R}^2)^V \rightarrow \mathbb{R}^E, ~ p \mapsto \left( \|p_v-p_w\|^2 \right)_{vw \in E}
	\end{equation*}
	at $p \in (\mathbb{R}^2)^V$ is a surjective $|E| \times 2|V|$ matrix if $\|p_v-p_w\| \neq 0$ for all $vw \in E$,
	and so the set $f_G^{-1}(1,\ldots,1)$ is a smooth manifold by the constant rank theorem (see for example \cite[Theorem 9.32]{rudin}).
	Hence the set $X$ -- an open subset of $f_G^{-1}(1,\ldots,1)$ -- is a smooth manifold.
	Define the smooth map
	\begin{equation*}
		f : X \rightarrow \mathbb{R}_{>0}^V, ~ p \mapsto \left( -1 + \sqrt{1 + \frac{4}{\|p_v\|^2}} ~ \right)_{v \in V}.
	\end{equation*}
	For any $p \in X$,
	the derivative $df(q)$ of $f$ at $p$ is a $|V| \times 2|V|$ surjective matrix.
	Hence the set $f(X)$ has an open interior by the constant rank theorem.
	
	Define the set
	\begin{equation*}
		Y := \Big\{ (p,r) : (G \oplus K_2,p,r) \text{ is a sphere packing}, p_a = (0,0,-1),p_b=(0,0,1),r_a=r_b=1   \Big\},
	\end{equation*}
	i.e., the set of sphere packings with contact graph $G \oplus K_2$ where the spheres corresponding to $a,b$ -- now fixed to be $S_a,S_b$ respectively -- have radius 1 and centres $(0,0,-1),(0,0,1)$ respectively. 
	Fix $g : \mathbb{R}^3 \dashrightarrow \mathbb{R}^3$ to be the M\"{o}bius transform where $g(x) = 2x/\|x\|^2$,
	and define $Z \subset \mathbb{R}^3 \times \mathbb{R}_{>0}$ to be the set of all spheres in contact with both $S_a$ and $S_b$.
	The map $g$ sends the spheres $S_a, S_b$ to the affine planes $\{ (x,y,-1) : x,y \in \mathbb{R}\}$ and $\{ (x,y,1) : x,y \in \mathbb{R}\}$ respectively.
	Hence the map $g$ takes any sphere $S$ that is in contact with both the spheres $S_a,S_b$,
	and outputs a unit sphere with centre $(x,y,0)$ for some $x,y$ satisfying $x^2 +y^2 > 1$.
	Given $\mathbb{D} := \{(x,y) : x^2+y^2 \leq 1\}$,
	define $k: Z \rightarrow \mathbb{R}^2 \setminus \mathbb{D}$ to be the bijective map which takes any sphere $S$ in contact with both $S_a,S_b$ with centre $(x,y,z)$ and outputs the point $(x',y') \in \mathbb{R}^2$ where $g(x,y,z) = (x',y',0)$;
	in essence, the map $k$ is mapping any sphere $S \in Z$ to a unit circle $C$ in $\mathbb{R}^2$ which does not contain the point $(0,0)$ inside its interior.
	Now define the map 
	\begin{equation*}
		h:Y \rightarrow X, ~ (p,r) \mapsto (k(p_v,r_v))_{v \in V},
	\end{equation*}
	and note that $h$ is bijective since $k$ is bijective.
	By our choice of map $f$ we have $f \circ h(p,r) = r$,
	and so $f(X) = \radii_{a=b=1}(G \oplus K_2)$.
	Thus $\radii(G \oplus K_2)$ has non-empty interior by \Cref{l:openset}.
	This now guarantees the existence of a sphere packing with contact graph $G \oplus K_2$ and generic radii.
\end{proof}

We conclude the section with a few corollaries of \Cref{t:main}.
The maximal clique size for the contact graph of a sphere packing is 5, and this bound is tight;
this follows from combining \Cref{t:penny} with the observation that $K_3$ is a penny graph but $K_4$ is not.
The following result is an easy corollary of \Cref{t:main} and the observation that $K_n = K_{n-2} \oplus K_2$ for every $n \geq 3$.

\begin{corollary}\label{cor:k5}
	The maximal clique size for the contact graph of a generic radii sphere packing is 4, and this bound is tight.
\end{corollary}

We recall that a graph is \emph{chordal} if every cycle of the graph that is an induced subgraph has length 3.

\begin{corollary}\label{cor:chordal}
	Let $P$ be a sphere packing with generic radii and contact graph $G$.
	If $G$ is chordal then $P$ is stress-free.
\end{corollary}

\begin{proof}
	Without loss of generality we may suppose $G$ is connected.	
	The result is immediately true if $|V|=1$.
	Suppose the result holds for any sphere packing with generic radii and chordal contact graph with less than $|V|$ vertices.
	By \Cref{cor:k5}, the maximal clique size of $G$ is 4 or less.
	Hence $G$ contains a vertex $v_0$ with neighbours $U \subset V$ such that $U$ is a clique of size at most 3.
	Fix $p:V \rightarrow \mathbb{R}^3$ and $r:V \rightarrow \mathbb{R}_{>0}$ to be the centres and radii of $P$ respectively.
	Following from our inductive hypothesis,
	the restriction of the framework $(G,p)$ to the vertices $V \setminus \{v_0\}$ is stress-free.
	By \Cref{l:0ext}, it is now sufficient to prove that the points $S := \{p_v: v \in U \cup \{v_0\}\}$ are affinely independent,
	which can be done using a similar technique as that presented in \Cref{l:stressfree}.
\end{proof}

We recall that a tree is a \emph{caterpillar graph} if a path remains
after all degree 1 vertices have been removed.

\begin{corollary}\label{cor:caterpillar}
	Let $G$ be a caterpillar graph.
	Then the following statements are equivalent.
	\begin{enumerate}
		\item \label{cor:caterpillar1} $G \oplus K_2$ is the contact graph of a sphere packing with generic radii.
		\item \label{cor:caterpillar2} The maximum degree of $G$ is 5, and any path between two degree 5 vertices must contain a vertex of degree 3 or less.
	\end{enumerate}
\end{corollary}

\begin{proof}
	A caterpillar graph is a penny graph if and only if condition \ref{cor:caterpillar2} holds \cite[Lemma 1]{KNP}.
	Equivalence now follows from \Cref{t:main}.
\end{proof}

\section{Complexity}

It follows from \Cref{t:main} that the problem of determining if a graph is the contact graph of a sphere packing with generic radii is as complex as the problem of determining if a tree is a penny graph.
It is believed that the problem of determining whether a tree is NP-Hard \cite{BDLRST}.
There is evidence to support this:
determining if a graph is a penny graph is NP-Hard \cite{BK96},
and determining if a tree with a fixed orientation is a penny graph is NP-Hard \cite{BDLRST}.
Because of this, we conjecture the following.

\begin{conjecture}\label{conj:complex}
	Determining whether a graph is the contact graph of a generic radii sphere packing is NP-Hard.
\end{conjecture}

To prove \Cref{conj:complex}, it would be sufficient to prove that the problem of determining whether a tree is NP-Hard.

\subsection*{Acknowledgements}
I would like to thank both Ethan Lee and Ross Paterson for their valuable discussions about algebraic functions and algebraic independence.
I would also like to thank Meera Sitharam for sharing her knowledge regarding circle and sphere packings, and Steven J.~Gortler for his helpful feedback on an earlier version of this paper.
Finally,
I would like to thank the anonymous reviewer who helped correct some of the technical lemmas in \Cref{sec:algfun} with their solid suggestions.

The author was supported by the Heilbronn Institute for Mathematical Research.

\begin{appendix}

\section{Algebraic functions}\label{app:algebraic functions}

This section contains some classical results regarding algebraic functions.

\begin{lemma}\label{l:algcoeff}
	Let $\mathbb{F}$ be a subfield of $\mathbb{R}$, let $D \subset \mathbb{R}^n$ be an open set and let $f:D \rightarrow \mathbb{R}$ be a continuous function.
	Suppose there exists algebraic functions $a_0,\ldots,a_m :D \rightarrow \mathbb{R}$ over $\mathbb{F}$ such that $a_k \neq 0$ for each $k\in \{0,\ldots,m\}$ and $\sum_{k=0}^m a_k(x) f(x)^k = 0$ for every $x \in D$.
	Then $f$ is also algebraic over $\mathbb{F}$.
\end{lemma}

\begin{proof}
	Denote the set of $n$-variable rational functions over $\mathbb{F}$ with domain restricted to $D$ by $\mathbb{K}$.\footnote{Technically each element of $\mathbb{K}$ is an equivalence class of rational functions that agree on a set $Z \cap D$,
	where $Z$ is Zariski open subset of $\mathbb{R}^n$.}
	The set $\mathbb{K}$ is a field with the pointwise addition and multiplication operations and the maps $x \mapsto 0$, $x \mapsto 1$ as its zero and unit elements respectively.
	Fix $\mathbb{L}$ to be the set of all possible partially-defined functions formed from finitely many elements of $\mathbb{K} \cup \{a_0,\ldots,a_m\}$ by pointwise addition and multiplication and taking reciprocals.
	Note that every element of $\mathbb{L}$ is a partially-defined algebraic function $g: D \dashrightarrow \mathbb{R}$,
	in the sense that $g$ is a well-defined algebraic function when restricted to an open dense set $Z \cap D$,
	where $Z$ is a Zariski open set.
	Since each $a_i$ is an algebraic function over $\mathbb{F}$,
	$\mathbb{L}$ is an algebraic field extension of $\mathbb{K}$.
	By our construction,
	$f$ is contained within an algebraic field extension of $\mathbb{L}$;
	this field can be constructed in the same way that we constructed $\mathbb{L}$ except with the initial set $\mathbb{K} \cup \{f,a_0,\ldots,a_m\}$.
	Hence $f$ is contained within an algebraic extension of $\mathbb{K}$ (see for example \cite{lang}).
	Thus there exists polynomial functions $p_0,\ldots,p_\ell \in \mathbb{K}$ such that $\sum_{k=0}^m p_k(x) f(x)^k = 0$ for every $x \in D$ as required.
\end{proof}

Given a subset $\{n_1,\ldots,n_k\} \subset \{1,\ldots,n\}$,
fix $\mathbb{A}_D(X_{n_1},\ldots,X_{n_k};\mathbb{F})$ to be the set of algebraic functions over the field $\mathbb{F}$ with domain $D$ that only take variables $X_{n_1},\ldots,X_{n_k}$,
i.e., algebraic functions $f:D \rightarrow \mathbb{R}$ that are invariant under changing the value of any coordinate $i \notin \{n_1,\ldots,n_k\}$.
For simplicity we set $\mathbb{A}_D(X_{n_1},\ldots,X_{n_k} ) = \mathbb{A}_D(X_{n_1},\ldots,X_{n_k};\mathbb{Q})$.
For example,
the function $f: \mathbb{R}^2_{>0} \rightarrow \mathbb{R}$ where $f(x_1,x_2) = \sqrt{x_1}$ lies in the set $\mathbb{A}_{\mathbb{R}^2_{>0}} (X_1)$.

For the following result we recall that the transcendence degree of a field extension $\mathbb{K}/\mathbb{F}$ (denoted $\trdeg(\mathbb{K}/\mathbb{F})$) is the cardinality of the largest set $S \subset \mathbb{K}$ that is algebraically independent over $\mathbb{F}$.
Note that a finite set $S \subset \mathbb{K}$ is algebraically dependent over $\mathbb{F}$ if and only if $\trdeg(\mathbb{F}(S) /\mathbb{F}) < |S|$.
Transcendence degrees satisfy the following equation:
if $\mathbb{F} \subset \mathbb{K} \subset \mathbb{L}$ are fields then
\begin{equation*}
	\trdeg(\mathbb{L}/\mathbb{K}) + \trdeg(\mathbb{K}/\mathbb{F}) = \trdeg(\mathbb{L}/\mathbb{F}).
\end{equation*}

\begin{lemma}\label{l:polyalgdep}
	Let $\mathbb{F} \subset \mathbb{R}$ be a subfield and let $p \in \mathbb{F}[X]$.
	Suppose that the set $\{x_1,\ldots,x_k\}$ is algebraically dependent over $\mathbb{F}$.
	Then the set $\{p(x_1),\ldots,p(x_k)\}$ is algebraically dependent over $\mathbb{F}$.
\end{lemma}

\begin{proof}
	Set $\mathbb{L} = \mathbb{F}(\{x_1,\ldots,x_k\})$ and $\mathbb{K} = \mathbb{F}(\{p(x_1),\ldots,p(x_k)\})$.
	Since $\mathbb{F} \subset \mathbb{K} \subset \mathbb{L}$ and the set $\{x_1,\ldots,x_k\}$ is algebraically dependent over $\mathbb{F}$,
	the following inequality holds:
	\begin{equation*}
		\trdeg(\mathbb{K}/\mathbb{F}) \leq \trdeg(\mathbb{L}/\mathbb{K}) + \trdeg(\mathbb{K}/\mathbb{F}) = \trdeg(\mathbb{L}/\mathbb{F}) < k.
\end{equation*}
	Hence $\{p(x_1),\ldots,p(x_k)\}$ is algebraically dependent over $\mathbb{F}$.
\end{proof}

\begin{lemma}\label{l:algfunalgdep}
	Let $\mathbb{F} \subset \mathbb{R}$ be a subfield, let $D \subset \mathbb{R}$ be an open set and let $f \in \mathbb{A}_D(X;\mathbb{F})$.
	If the set $\{x_1,\ldots,x_k\} \subset D$ is algebraically dependent over $\mathbb{F}$ then the set $\{f(x_1),\ldots,f(x_k)\}$ is algebraically dependent over $\mathbb{F}$.
\end{lemma}

\begin{proof}
	Fix polynomials $p_0,\ldots,p_m \in \mathbb{F}[X]$ such that $\sum_{i=0}^m p_i(x) f(x)^i = 0$ for all $x \in D$.
	Hence $f(x_i)$ lies in the algebraic closure of $\mathbb{F}(\{x_i\})$ for each $i \in \{1,\ldots,k\}$.
	Set $\mathbb{L} = \mathbb{F}(\{x_1,\ldots,x_k\})$ and $\mathbb{K} =\mathbb{F}(\{f(x_1),\ldots,(x_k)\})$,
	and denote the algebraic closure of $\mathbb{L}$ by $\overline{\mathbb{L}}$.
	As $\trdeg(\overline{\mathbb{L}}/\mathbb{L})= 0$ and $\{x_1,\ldots,x_k\}$ is algebraically dependent over $\mathbb{F}$,
	we have that $\trdeg(\overline{\mathbb{L}}/\mathbb{F})<k$.
	As $\mathbb{K} \subset \overline{\mathbb{L}}$,
	the following equality holds:
	\begin{equation*}
		\trdeg(\mathbb{K}/\mathbb{F}) \leq \trdeg(\overline{\mathbb{L}}/\mathbb{K}) + \trdeg(\mathbb{K}/\mathbb{F}) = \trdeg(\overline{\mathbb{L}}/\mathbb{F}) < k.
	\end{equation*}
	Hence $\{f(x_1),\ldots,f(x_k)\}$ is algebraically dependent over $\mathbb{F}$.
\end{proof}

\begin{lemma}\label{l:algfuninv}
	Let $\mathbb{F} \subset \mathbb{R}$ be a subfield, let $D_1,D_2 \subset \mathbb{R}$ be open sets and let $f: D_1 \rightarrow D_2$ be a bijective algebraic function over $\mathbb{F}$.
	Then the inverse $f^{-1} :D_2 \rightarrow D_1$ is an algebraic function over $\mathbb{F}$.  
\end{lemma}

\begin{proof}
	Fix $p \in \mathbb{F}[X,Y]$ to be the non-zero polynomial where $p(f(x),x)=0$ for each $x \in D_1$.
	Hence if we pick any $y \in D_2$ and apply the substitution $x = f^{-1}(y) \in D_1$ to our previous equation we see that $p(y,f^{-1}(y))=0$ as required.
\end{proof}

\end{appendix}

\end{document}